\documentclass[11pt]{amsart}

\usepackage{amsmath}
\usepackage{amssymb}
\usepackage{mathrsfs}
\usepackage[all]{xy}

\newtheorem{theorem}{Theorem}[section]
\newtheorem{claim}[theorem]{Claim}

\newtheorem{lemma}[theorem]{Lemma}

\newtheorem{proposition}[theorem]{Proposition}
\newtheorem{corollary}[theorem]{Corollary}

\newtheorem*{theorem*}{Theorem}
\newtheorem*{corollary*}{Corollary}

\theoremstyle{definition}
\newtheorem{definition}[theorem]{Definition}

\newtheorem{question}[theorem]{Question}

\theoremstyle{remark}
\newtheorem{remark}[theorem]{Remark}

\newtheorem{fact}[theorem]{Fact}

\def\l{{\langle}}
\def\r{{\rangle}}

\newcount\skewfactor
\def\mathunderaccent#1#2 {\let\theaccent#1\skewfactor#2
\mathpalette\putaccentunder}
\def\putaccentunder#1#2{\oalign{$#1#2$\crcr\hidewidth
\vbox to.2ex{\hbox{$#1\skew\skewfactor\theaccent{}$}\vss}\hidewidth}}

\def\smallbox#1{\leavevmode\thinspace\hbox{\vrule\vtop{\vbox
   {\hrule\kern1pt\hbox{\vphantom{\tt/}\thinspace{\tt#1}\thinspace}}
   \kern1pt\hrule}\vrule}\thinspace}

\DeclareMathOperator{\rng}{rng}
\newcommand{\cf}{{\rm cf}}

\title{Scales in the point spectrum}
\author{Tom Benhamou}
\thanks{The research of the first author was supported by the National Science Foundation under Grant
No. DMS-2346680}
\address[Benhamou]{Department of Mathematics, Rutgers University, New Brunswick, NJ, USA}
\email{tom.benhamou@rutgers.edu}

\begin{document}
\begin{abstract}
     We study the Point/Tukey spectrum of a general directed set using PCF theoretic tools and uncover basic connections between the theories. In particular, we prove that if the supremum of the Tukey spectrum is singular, then its cofinality must also be a member of the Tukey spectrum.
\end{abstract}
\maketitle
\section{Introduction}
Consider a directed set $\mathbb{P}=(P,\leq_P)$.  A prominent quantity which is attached to $\mathbb{P}$ is the \textit{character of $\mathbb{P}$} (also known as the \textit{cofinality} of $\mathbb{P}$) which is defined as follows:
$$\chi(\mathbb{P})=\min\{|\mathcal{B}|\mid \mathcal{B}\subseteq P, \ \forall p\in P\exists b\in \mathcal{B}, \ p\leq_Pb\}$$
A set $\mathcal{B}\subseteq P$ satisfying the above condition: $\forall p\in P\exists b\in\mathcal{B},p\leq_P b$ is called \textit{a $\mathbb{P}$-cofinal family}\footnote{We omit $\mathbb{P}$ when there is no risk of ambiguity.}. Problems which involved ordered sets often reduce to the study of cofinal families and characters. For example in Topology, first countable spaces are those topological spaces such that the filter of open neighborhoods $\mathcal{N}(x)$ of every point $x$ admits a countable filter bases, which is translated to requiring that $\chi(\mathcal{N}(x),\supseteq)\leq\omega$. More generally, the study of $(\mathcal{N}(x),\supseteq)$-cofinal families is intrinsically linked to the type of nets needed in order to characterize continuity in terms of sequencial-like continuity in the context of Moore-Smith convergence. The main tool to compare the complexity of cofinal families is the \textit{Tukey order} \cite{Tukey40} which is defined as follows: given two directed sets $\mathbb{P}=(P,\leq_P)$ and $\mathbb{Q}=(Q,\leq_Q)$ we say that $\mathbb{P}\leq_T\mathbb{Q}$ if there is a cofinal map $f:Q\to P$, that is, every $\mathbb{Q}$-cofinal family is mapped by $f$ to a $\mathbb{P}$-cofinal family. It is easy to see that the character is an invariant of the Tukey order, namely, if $\mathbb{P}\leq\mathbb{Q}$ then $\chi(\mathbb{P})\leq\chi(\mathbb{Q})$. We say that $\mathbb{P}\equiv_T\mathbb{Q}$ if $\mathbb{P}\leq_T\mathbb{Q}$ and $\mathbb{Q}\leq_T\mathbb{P}$. We call the $\equiv_T$-equivalent class of a poset $\mathbb{P}$ the \textit{Tukey type} or \textit{cofinal type} of $\mathbb{P}$. The Tukey type of poset of the form $(\mathcal{U},\supseteq)$, where $\mathcal{U}$ is an ultrafilters has been  a fruitful line of research \cite{Milovich08,Dobrinen/Todorcevic11,Raghavan/Todorcevic12} which is intensively studied in recent days \cite{Benhamou/Dobrinen24,BENHAMOU_DOBRINEN_2024,CancinoZaplatal,Benhamou_Goldberg_2025}. Yet, some fundamental questions such as the \textit{kunen problem} regarding the possible values of the character of an ultrafilter over $\omega_1$ remain open. When considering ultrafilters, a natural perspective of certain problems is the through the ultrapower construction, which encapsulates many of the combinatorial properties of the ultrafilter and provides a clear vision of certain properties of the ultrafilter e.g. completness, regularity, indecomposability, $p$-point, $q$-point etc. Given an ultrafilter $\mathcal{U}$ over a set $X$, the \textit{ultrapower of $V$ by $\mathcal{U}$} is denoted by $M_{\mathcal{U}}=V^X/\mathcal{U}$ and consist of all $=_{\mathcal{U}}$-equivalence classes  $[f]_{\mathcal{U}}$ of function $f:X\to V$, and the \textit{ultrapower embedding} $j_{\mathcal{U}}:V\to M_{\mathcal{U}}$ is defined by $j_{\mathcal{U}}(x)=[c_x]_{\mathcal{U}}$, where $c_x$ is the constant function with value $x$.
One of the motivating questions of this paper is the following:
\begin{question}\label{Queation1}
    What is the meaning of $\chi(\mathcal{U})$ in terms of $(M_{\mathcal{U}},j_{\mathcal{U}})$?
\end{question} To be more precise, the question is whether there is a reasonable way of expressing the quantity $\chi(\mathcal{U})$ via objects related to the ultrapower $(M_{\mathcal{U}},j_{\mathcal{U}})$. Heuristically, such a characterization ought to exists because of the following reason: recall that two ultrafilters $\mathcal{U},\mathcal{W}$ over $X,Y$ respectively are called \textit{Rudin-Keisler} equivalent and denoted by $\mathcal{U}\simeq _{RK}W$, if there $\varphi:X\to Y$ such that for some $A\in \mathcal{U}$, $\varphi\restriction A$ is one-to-one, and $\mathcal{W}=\{B\subseteq Y\mid \varphi^{-1}[B]\in \mathcal{U}\}$. It is well known that $\mathcal{U}\simeq_{RK}\mathcal{W}$ if and only if there is an isomorphism $k:M_{\mathcal{U}}\simeq M_{\mathcal{W}}$ such that $j_{\mathcal{W}}=k\circ j_{\mathcal{U}}$. Since Rudin-Keisler equivalence implies Tukey-equivalence (see for example \cite{DobrinenTukeySurvey15}), and since $\chi(\mathcal{U})$ is an invariant of the Tukey-equivalence, it follows that $\chi(\mathcal{U})$ indeed only depend on the $RK$-equivalent class of $\mathcal{U}$, or equivalently, the isomorphism class of $(M_{\mathcal{U}},j_{\mathcal{U}})$. In this paper, we proved evidence that such a characterization might relate to the notion of the \textit{Tukey spectrum} which is defined for an arbitrary directed set $\mathbb{P}$.

While determining the Tukey type of a directed set $\mathbb{P}$ might be considered a hard task (or even worst, its value can be independent of $\mathsf{ZFC}$), there are several yardstick cofinal types which are used as measures of cofinal complexity. These have been called in the past \textit{standard Tukey sets}  \cite{Isbell65,Schmidt55} and are defined to be posets of the form $([\lambda]^{<\mu},\subseteq)$ where $\mu\leq\lambda$ are cardinals and $[\lambda]^{<\mu}=\{X\subseteq\lambda\mid |X|<\mu\}$. It is well known that  if $\chi([\lambda]^{<\mu},\subseteq)=\lambda$. Then the following are equivalent:
    \begin{enumerate}
        \item $([\lambda]^{<\mu},\subseteq)\leq_T\mathbb{P}$.
        \item for any ${<}\mu$-directed poset $\mathbb{Q}$, $\mathbb{Q}\leq_T \mathbb{P}$.
    \end{enumerate}
In this paper we will mostely be interested in stadard Tukey sets in the case where $\lambda=\mu$ is a regular cardinal, in which case $([\lambda]^{<\lambda},\subseteq)\equiv_T (\lambda,<)$, where $<$ induces the usual well-order of $\lambda$ (i.e. $<$ is just $\in$). Hence the following sets a measure of cofinal complexity:
\begin{definition}
    The Tukey (or point) spectrum of $\mathbb{P}$ is defined by $$Sp_T(\mathbb{P})=\{\lambda\in Reg\mid \lambda\leq_T \mathbb{P}\}$$
\end{definition}
The higher $\mathbb{P}$  is in the Tukey order, the reacher is its Tukey spectrum. This notion has been studied in several occasions \cite{Isbell65,Gartside/Mamatelashvili15,TomCohesive,Gilton2024,MamatelashviliThesis}. The basic relation between $\chi(\mathbb{P})$ and $Sp_T(\mathbb{P})$ is that $\sup(Sp_T(\mathbb{P}))\leq \chi(\mathbb{P})$, while $cf(\chi(\mathbb{P}))\in Sp_T(\mathbb{P})$.
As a corollary we immediately see that if $\chi(\mathbb{P})$ is a regular cardinal, then $\chi(\mathbb{P})=\sup(Sp_T(\mathbb{P}))$.
The reason that the Tukey spectrum is helpful here is due to \cite{TomCohesive}, where the following characterization of $Sp_T(\mathcal{U})$ in terms of $(M_{\mathcal{U}},j_{\mathcal{U}})$ was given:
\begin{theorem}
    Let $\mathcal{U}$ be an ultrafilter and $\lambda$ be a regular cardinal. Then the following are equivalent:
    \begin{enumerate}
        \item $\lambda\in Sp_T(\mathcal{U})$.
        \item There is $X\in M_{\mathcal{U}}$ such that:
        \begin{enumerate}
            \item $X$ covers $j_{\mathcal{U}}''\lambda$ i.e. for every $\alpha<\lambda$, $M_{\mathcal{U}}\models j_{\mathcal{U}}(\alpha)\in X$.
            \item $X$ is $\lambda$-thin i.e. for every $I\in [\lambda]^\lambda$, $M_{\mathcal{U}}\models j_{\mathcal{U}}(I)\not\subseteq X$.
        \end{enumerate}
    \end{enumerate}
\end{theorem}
\begin{corollary}
    If $\chi(\mathcal{U})=\sup(Sp_T(\mathcal{U}))$, then $\chi(\mathcal{U})$ is the the supremum of all regular $\lambda$'s for which $M_{\mathcal{U}}$ has a $\lambda$-this cover for $j_{\mathcal{U}}''\lambda$.
\end{corollary}
In particular if $\chi(\mathcal{U})$ is regular then we obtain an ultrapower characterization of it in terms of $(M_{\mathcal{U}},j_{\mathcal{U}})$. This brings us to our second central question of this paper:
\begin{question}\label{question2}
    Is it $\mathsf{ZFC}$-provable that $\sup(Sp_T(\mathbb{P}))=\chi(\mathbb{P})$? can we prove it specifically for ultrafilters?
\end{question}
To answer this question, we are left to deal with the case where $\chi(\mathbb{P})$ is singular. Isbell \cite{Isbell65} gave a positive answer under cardinal arithmetic assumptions:
\begin{theorem}[Isbell]\label{Isbells Thm}
    Suppose that every singular $\lambda\leq\chi(\mathbb{P})$ is a strong limit cardinal. Then $\chi(\mathbb{P})=\sup(Sp_T(\mathbb{P}))$. 
\end{theorem}
In this paper, without any cardinal arithmetic assumption, we provide  several $\mathsf{ZFC}$-restrictions on $\sup(Sp_T(\mathbb{P}))$. 
The proof establishes a connection to another main branch of the study of cofinalities-- PCF theory. While some connections betweent the Tukey order and PCF theory have been studied in the past by Gartside- Mamatelashvili \cite{Gartside/Mamatelashvili15}, Gilton \cite{Gilton2024} and very recently by Gartside-Gilton \cite{GarsideGilton}, the line of research pursued in this paper are quite different; we demonstrate how PCF theory can be used to obtain new $\mathsf{ZFC}$-restrictions on the structure of $Sp_T(\mathbb{P})$ for a general directed set $\mathbb{P}$. 
\begin{theorem*}
    Suppose that $\lambda$ is a limit cardinal and  $B\subseteq Sp_T(\mathbb{P})\cap\lambda$ are such that:
    \begin{enumerate}
        \item [(i)] $B$ is unbounded in $\lambda$.
        \item [(ii)] $cf(|B|)=|B|\notin Sp_T(\mathbb{P})$.
        \item [(iii)] $tcf(\prod B/\mathcal{J}^B_{bd})=\mu$.
    \end{enumerate}  
Then $\mu\in Sp_T(\mathbb{P})$.
\end{theorem*}
In particular we improve several results from \cite{Isbell65} and show that the cofinality of $\sup(Sp_T(\mathbb{P}))$ behaves similarly to the cofinality of $\chi(\mathbb{P})$ in the following sense:
\begin{corollary*}
    Suppose that $\sup(Sp_T(\mathbb{P}))$ is singular. Then $cf(\sup(Sp_T(\mathbb{P})))\in Sp_T(\mathbb{P})$.
\end{corollary*}
This also allows to transfer results from traditionary PCF theory to results about the point spectrum:
\begin{corollary*}
     Suppose $\lambda$ is a singular limit cardinal such that $Sp_T(\mathbb{P})\cap \lambda$ is unbounded in $\lambda$.
    \begin{enumerate}
        \item  $Sp_T(\mathbb{P})$ contains a cardinal in $\{cf(\lambda)\}\cup Pcf(Sp_T(\mathbb{P})\cap\lambda)$. 
        \item If $\mathbb{P}$ is $cf(\lambda)^+$-directed, then $\lambda^+\in Sp_T(\mathbb{P})$.
        \item If $cf(\lambda)\notin Sp_T(\mathbb{P})$ and $Sp_T(\mathbb{P})$ contains a tail of regular cardinals (or just $C^{(+)}$ for some club $C$ in $\lambda$) then $\lambda^+\in Sp_T(\mathbb{P})$.
    \end{enumerate} 
\end{corollary*}
In the case that $\lambda$ is an inaccessible limit of $Sp_T(\mathbb{P})$, we have the following corollary:
\begin{corollary*}
    Suppose that $\lambda$ is weakly inaccessible and $B\subseteq Sp_T(\mathbb{P})\cap\lambda$ is unbounded and  $tcf(\prod B/\mathcal{J}^\lambda_{bd})$ exists. $Sp_T(\mathbb{P})\cap[\lambda,2^\lambda)\neq\emptyset$. 
\end{corollary*}
The results above can also be seen as \textit{stepping up} type of results for limit steps of $Sp_T(\mathbb{P})$, in attempt to provide a positive answer to Question \ref{question2}. Unfortunately, our results give no information about the case where $\max(Sp_T(\mathbb{P}))$ is a successor cardinal.

In the second part of this paper we provide a new strategy to address Question \ref{Queation1} by considering two generalizations. In the first, we study some ideals which naturally arise from the order of $\mathbb{P}$ and demonstrate how the PCF results can be proved in a more general set-up:
\begin{theorem*}
    Let $\mathbb{P},\mathbb{Q}$ and $(\mathbb{R}^p)_{p\in\mathbb{P}}$ we a collection of directed sets such that:
    \begin{enumerate}
        \item $\mathbb{P}\not\leq_T\mathbb{Q}$.
        \item For each $p\in \mathbb{P}$, $\mathbb{R}^p\leq_T\mathbb{Q}$.
    \end{enumerate} There is a natural map $F:\prod_{p\in\mathbb{P}}\mathbb{P}^p/J_{bd}(\mathbb{P})\to\mathbb{Q}$ mapping a cofinal subset of $(\prod_{p\in\mathbb{P}}\mathbb{P}^p/J_{bd}(\mathbb{P}),\leq_{J_{bd}(\mathbb{P})})$ to an unbounded subset of $\mathbb{Q}$. In particular if $\prod_{p\in\mathbb{P}}\mathbb{P}^p/J_{bd}(\mathbb{P})$ has a linear cofinal subset, then $\prod_{p\in\mathbb{P}}\mathbb{P}^p/J_{bd}(\mathbb{P})\leq_T\mathbb{Q}$
\end{theorem*}
In the second generalization we introduce the $\mathcal{I}$-cohesive property for some ideal $\mathcal{I}$  over $\lambda$, such that:
\begin{enumerate}
    \item If $\lambda\in Sp_T(\mathbb{P})$ then for any $\mathcal{I}$ extending the $\mathcal{J}^\lambda_{bd}$, $\mathbb{P}$ is $\mathcal{I}$-cohesive.
    \item $\mathcal{I}$-cohesivness has an ultrapower characterization when $\mathbb{P}=\mathcal{U}$ is an ultrafilter.
    \item If $\mathbb{P}$ is $\mathcal{I}$-cohesive for some $\lambda$-complete ideal $\mathcal{I}$ over $\lambda$, then $\lambda\leq \chi(\mathbb{P})$. 
\end{enumerate}
We then prove a strengthening of the previous stepping up results:
\begin{theorem*}
    Suppose that $B\subseteq Sp_T(\mathbb{P})$ and $cf(|B|)=|B|$ and let $\mathcal{I}$ be an ideal on $|B|$ such that:
    \begin{enumerate}
        \item  $\mathbb{P}$ is not $\mathcal{I}$-cohesive.
        \item $tcf(\prod B /\mathcal{I})=\mu$.
    \end{enumerate}
    Then $\mu\in Sp_T(\mathbb{P})$.
\end{theorem*}
\section{On the Tukey spectrum and PCF theory}
Let us start with some preliminary results regarding the Tukey spectrum. After proving Theorem \ref{Isbells Thm}, Isbell provided several other results regarding limit points of the point spectrum.
\begin{theorem}[Isbell]\label{thm more isbell}
    Let $\mathbb{P}$ be a directed set and let $\lambda$ be a limit point of $Sp_T(\mathbb{P})$. 
    \begin{enumerate}
        \item There is $\lambda\leq \mu\leq 2^\lambda$ such that $cf(\mu)\in Sp_T(\mathbb{P})$.
        \item If $Sp_T(\mathbb{P})\cap\lambda$ is stationary in $\lambda$ (in particular $cf(\lambda)>\aleph_0$), then $cf(\lambda)\in Sp_T(\mathbb{P})$.
        \item If $\lambda$ is singular, and for a stationary set of $m<\lambda$, $\{cf(\mu)\mid \mu\in(m,2^m]\}\cap Sp_T(\lambda)=\emptyset$ then $cf(\lambda)\in Sp_T(\mathbb{P})$.
    \end{enumerate}
\end{theorem}
Isbell's result is obtained by combining monotone cofinal maps from $\mathbb{P}$. Namely, suppose that for each $\alpha$, $\varphi_\alpha:\mathbb{P}\to \mathbb{Q}_\alpha$ is a monotone and cofinal map. Then Isbell defines $\varphi^*:\mathbb{P}\to \prod_{\alpha}\mathbb{Q}_\alpha$ naturally $\varphi^*(p)=(\varphi_\alpha(p))_\alpha$. The map $\varphi^*$ has three crucial properties:
\begin{enumerate}
    \item It is monotone.
    \item $E:=\rng(\varphi^*)$ is directed.
    \item $\pi_{\alpha}\circ\varphi^*=\varphi_\alpha$, where $\pi_\alpha$ is the projection to the $\alpha^{\text{th}}$ coordinate.
    \item $E\leq_T \mathbb{P}$.
\end{enumerate}
The set $E$ is called a \textit{subdirect product}. Then Theorem \ref{thm more isbell} follows by analyzing $E$ in the case where $\mathbb{Q}_\alpha\in Sp_T(\mathbb{P})$. For example the $\mu$ witnessing $(1)$ is just $\chi(E)$ (indeed $\lambda\leq\chi(E)\leq |E|\leq 2^\lambda$). Also $(2)$ follows by defining a monotone cofinal map from $E$ to $\lambda$, which takes $e\in E$ and outputs the minimal $i<\lambda$ such that for unboundedly many  $\alpha$, $e(\alpha)<i$. The properties above guarantee that this is a monotone and cofinal map. The disadvantage of this method is that we do not have much control on the set $E$ and how it relates to the full product $\prod_\alpha\mathbb{Q}_\alpha$. 

In this paper we are mostely interested in the dual of this method. Namely, given $\psi_\alpha:\mathbb{Q}_\alpha\to \mathbb{P}$ which are unbounded, combining these maps into a single $\psi$ which is defined on interesting sets.  
The main source of such sets comes from PCF theory. This will ultimately result in improvment of the above result. Before moving forward in that direction, let us start with a simple case, when $\mathbb{P}$ is sufficiently directed:
\begin{theorem}\label{Simple case}
    Suppose that $(\mathbb{Q}_\alpha)_{\alpha<\lambda}$ is a collection of directed sets such that for every $\alpha<\lambda$, $\mathbb{Q}_\alpha\leq_T \mathbb{P}$. Suppose moreover that $\mathbb{P}$ is $\lambda^+$-directed, then $\prod_{\alpha<\lambda}\mathbb{Q}_\alpha\leq_T\mathbb{P}$.
\end{theorem}
\begin{proof}
    Let $\psi_\alpha:\mathbb{Q}_\alpha\to\mathbb{P}$ be unbounded maps. Let $(q_\alpha)_{\alpha<\lambda}\in \prod_{\alpha<\lambda}\mathbb{Q}_\alpha$, and consider 
    $(\psi_{\alpha}(q_\alpha))_{\alpha<\lambda}\subseteq\mathbb{P}$. Since $\mathbb{P}$ is $\lambda^+$-directed, there is $p$ such that for every $\alpha$, $\phi_\alpha(q_\alpha)\leq p$. 
    Set $\psi((q_\alpha)_{\alpha<\lambda})=p$. Let us argue that $\psi:\prod_{\alpha<\lambda}\mathbb{Q}_\alpha\to\mathbb{P}$ is unbounded witnessing $\prod_{\alpha<\lambda}\mathbb{Q}_\alpha\leq_T\mathbb{P}$. 
    Indeed, if $\mathcal{A}\subseteq\prod_{\alpha<\lambda}\mathbb{Q}_\alpha$ is unbounded, then for some $\alpha$, $\mathcal{A}\restriction\{\alpha\}=\{q_\alpha\mid \vec{q}\in\mathcal{A}\}$ in unbounded in $\mathbb{Q}_\alpha$. Since $\psi_\alpha$ is unbounded in $\mathbb{Q}_\alpha$, $\psi_\alpha[\mathcal{A}\restriction\{\alpha\}]$ is unbounded in $\mathbb{P}$. 
    Since for each $\vec{q}$, $\psi(\vec{q})\geq q_\alpha$, we conclude that also $\{\psi(\vec{q})\mid \vec{q}\in \mathcal{A}\}$ must also be unbounded in $\mathbb{P}$.
\end{proof}

Next, let us introducing some notations from Pcf theory. 
Let $B$ be a set of regular cardinal without a maximal element. In this section we consider $\prod B$-- the set of all function $f:B\to \sup(B)$ such that for every $b\in B$, $f(b)<b$. Given an ideal $\mathcal{J}$ on $B$, we consider $\prod B/\mathcal{J}$ ordered by $f\leq_{\mathcal{J}}g$ if and only if $\{b\in B\mid f(b)> g(b)\}\in \mathcal{J}$. $f<_{\mathcal{J}}g$ is defined similarily. A \textit{scale} in $\prod B/\mathcal{J}$ is sequence $\l f_\alpha\mid \alpha<\lambda\r$ which is $<_{\mathcal{J}}$-increasing and $<_{\mathcal{J}}$-cofinal.
A scale does not always exist. If $\mathcal{J}$ is a maximal ideal (i.e. the dual filter is an ultrafilter), then a scale exist. Let us say that the \textit{true cofinality} $tcf(\prod B/\mathcal{J})$ is equal to $ \mu$ if $\mu$ is the least length of a scale in $\prod B/\mathcal{J}$.
\begin{remark}
    Let $\mathfrak{b}(\prod B/\mathcal{J})$ be the minimal size of a $\leq_{\mathcal{J}}$-unbounded set of functions $\mathcal{F}\subseteq \prod B$ and let $\mathfrak{d}(\prod B/\mathcal{J})$ be the minimal size of a $\leq_{\mathcal{J}}$-cofinal family $\mathcal{F}\subseteq \prod B$. The (existence) statement $tcf(\prod B/\mathcal{J})=\mu$ is equivalent to $\mathfrak{b}(\prod B/\mathcal{J})=\mu=\mathfrak{d}(\prod B/\mathcal{J})$. 
\end{remark}
For a set $A$ of regular cardinals we define
$$Pcf(A)=\{cf(\prod A/U)\mid U\text{ is an ultrafilter over }A\}$$ For more information on PCF theory we refer the reader to \cite{AbrahamMagidor2010,ShelahPCF}.
\begin{corollary}
    Suppose that $B\subseteq Sp_T(\mathbb{P})$ and $\mathbb{P}$ is $|B|^+$-directed. Then $Pcf(B)\subseteq Sp_T(\mathbb{P})$.
\end{corollary}
\begin{proof}
    By Theorem \ref{Simple case}, $\prod B\leq_T \mathbb{P}$, and it is easy to see that $Pcf(B)\subseteq Sp_T(\prod B)\subseteq Sp_T(\mathbb{P})$.
\end{proof}

The following proposition provides the  main tool to study the limit points of the spectrum above the directedness degree of a directed set $\mathbb{P}$. We will need to analyze scales modulo the bounded ideal $\mathcal{J}^\lambda_{bd}$-- the ideal consisting of all bounded subsets of $\lambda$. 

\begin{proposition}\label{Main tool}
    Suppose that $\lambda$ is in $acc(Sp_T(\mathbb{P}))$ and that there is $B\subseteq Sp_T(\mathbb{P})$ such that:
    \begin{enumerate}
        \item [(i)] $B$ is unbounded in $\lambda$.
        \item [(ii)] $cf(|B|)=|B|\notin Sp_T(\mathbb{P})$.
        \item [(iii)] $tcf(\prod B/\mathcal{J}^\lambda_{bd})=\mu$.
    \end{enumerate}  
Then $\mu\in Sp_T(\mathbb{P})$.
\end{proposition}
\begin{proof}
       For each $\rho\in B$, let $\l p_{\alpha,\rho}\mid \alpha<\rho\r$ witness that $\rho\in Sp_T(\mathbb{P})$ that is:
       \begin{equation}\label{equation Tukey}
           \forall S\in [\rho]^\rho, \ \{p_{\alpha,\rho}\mid \alpha\in S\}\text{ is unbounded in }\mathbb{P}.
       \end{equation} Let $\l f_\beta \mid \beta<\mu\r$ be a scale modulo $\mathcal{J}^B_{bd}$. For every $\beta<\mu$, consider $\{p_{f_\beta(\rho),\rho}\mid \rho\in B\}$. Note that $f_\beta(\rho)<\rho$ hence $p_{f_\beta(\rho),\rho}$ is well defined. Since $|B|\notin Sp_T(\mathbb{P})$, there is $I_\beta\in [B]^{|B|}$ (in particular $I_\beta$ is unbodned in $\lambda$) such that $\{p_{f_\beta(\rho),\rho}\mid \rho\in I_\beta\}$ is bounded by some $p^*_\beta$. 
We claim that $\l p^*_\beta\mid \beta<\mu\r$ witness that $\mu \in Sp_T(\mathbb{P})$. Otherwise, there is $I\in \mu$ such that $\{p^*_\beta\mid \beta\in I\}$ is bounded by $p^*$.

\begin{claim}
    There is $\rho\in B$ such that $\{f_\beta(\rho)\mid \beta\in I, \rho\in I_\beta\}$ is unbounded in $\rho$. 
\end{claim}
\begin{proof}
    Otherwise, for every $\rho\in B$, $\{f_\beta(\rho)\mid \beta\in I, \rho\in I_\beta\}$ is bounded by $\alpha_\rho<\rho$. The function $f(\rho)=\alpha_\rho$ is then bounded modulo $\mathcal{J}^B_{bd}$ by some $f_\beta$ for $\beta<\mu$ and since $I$ is unbounded in $\mu$, we may assume that $\beta\in I$.  Namely, $f<_{\mathcal{J}^B_{bd}} f_{\beta}$. Since $I_{\beta}$ is unbounded in $\lambda$, there is $\rho_{\beta}\in I_{\beta}$ such that $f(\rho_{\beta})<f_{\beta}(\rho_{\beta})$. 

    But then $\beta\in I$ and $\rho_\beta\in I_\beta$ and therefore $f_\beta(\rho_\beta)\leq \alpha_\beta=f(\rho_\beta)<f_\beta(\rho_\beta)$, contradiction.
\end{proof}
Fix $\rho$ as in the claim. 
Then by (\ref{equation Tukey}), $\{p_{f_{\beta}(\rho),\rho}\mid \beta\in I,\rho\in I_\beta\}$ must be unbounded. On the other hand, it is bounded by $p^*$. Contradiction.
    
\end{proof}
To show some specific cases of interest. Let us start with singular cardinals and some basic PCF theory. Recall that for a set $A$ of regular cardinals and for every cardinal $\rho$, we define the ideal $\mathcal{J}_{<\rho}^A$ to consist of all $X\subseteq A$ such that $Pcf(X)\subseteq \rho$. Equivalently, for every ultrafilter $U$ over $A$ such that $X\in U$, $cf(\prod A/U)<\rho$. Here are some properties of $\mathcal{J}^{A}_{<\rho}$ which can be found in \cite{AbrahamMagidor2010}:
\begin{enumerate}
    \item $\mathcal{J}^A_{<\rho}$ is increasing and continuous with $\rho$.
    \item $\bigcup_{\rho}\mathcal{J}^A_{<\rho}=P(A)$.
\end{enumerate}
\begin{theorem}[PCF generators]
Let $A$ be a progressive set (i.e. $|A|<\min(A)$). There are sets $\{B^A_\theta\mid\theta\in Pcf(A)\}$ such that:
\begin{enumerate}
        \item $\mathcal{J}^A_{<\rho}$ is generated by $\{B^A_\theta\mid \theta\in Pcf(A)\cap\rho\}$.
        \item If $\theta\in Pcf(A)$ then  $tcf(\prod B^A_\theta/\mathcal{J}^A_{<\theta})=\theta$.
    \end{enumerate}
\end{theorem}
Using the PCF generator we can prove the following lemma. The author would like to thank James Cummings for showing him the proof.
\begin{lemma}
    Let $A$ be a progressive set. Then there is $\mu\in Pcf(A)\setminus \sup(A)$ such that $B^A_\mu$ is unbounded in $\sup(A)$ and $\mathcal{J}^A_{<\mu}\subseteq \mathcal{J}^{\sup(A)}_{bd}$. 
\end{lemma}
\begin{proof}
    
Let $\mu\in Pcf(A)$ be minimal such that $B^A_\mu$ is unbounded in $\sup(A)$.  
    Then by property $(1)$ of the PCF generators, $\mathcal{J}_{<\mu}\subseteq \mathcal{J}^{\sup(A)}_{bd}$, and by $(2)$ there is a scale $\l f_\alpha\mid \alpha<\mu\r$ in $\prod B^A_\mu/\mathcal{J}_{<\mu}$. Let us argue that $\mu\geq \sup(A)$. Otherwise, since $\mathcal{J}_{<\mu}\subseteq \mathcal{J}^\lambda_{bd}$, $B^A_\mu\setminus \mu+1\in \mathcal{J}_{<\mu}^+$. Define for $b\in B^A_\mu\setminus \mu+1$  $g(b)=\sup_{i<\mu}f_i(b)<b$ or $0$ otherwise. Then for every $i<\mu$, $g(b)>f_i(b)$ on a positive set which in turn implies that $g\not\leq_{\mathcal{J}_{<\mu}}f_i$, contradiction. 
\end{proof}
The following corollary improves Isbell's Theorem \ref{thm more isbell} items $(1)$.
\begin{corollary}
    Suppose that $\lambda\in acc(Sp_T(\mathbb{P}))$ is such that $\lambda>\cf(\lambda)\notin Sp_T(\mathbb{P})$. Then $Sp_T(\mathbb{P})\cap(\lambda,2^\lambda]\neq \emptyset$. 
\end{corollary}
\begin{proof}
    By the assumption on $\lambda$, there is a progressive set $A\subseteq Sp_T(\mathbb{P})$ such that $\text{otp}(A)=\cf(\lambda)$ and $\sup(A)=\lambda$, Then by the previous lemma, there is $B\subseteq A$ unbounded in $\lambda$, and $\mu\in Pcf(A)\setminus\lambda$ such that $\mathcal{J}^A_{<\mu}\subseteq \mathcal{J}^\lambda_{bd}$. Note that $\text{otp}(B)=cf(\lambda)$ and $\mu>\lambda$ (since $\mu$ is regular while $\lambda$ is singular). By the properties of PCF generators,
    $tcf(\prod B^A_\mu/\mathcal{J}^A_{<\theta})$ exists and equals $\mu$. Hence $tcf(\prod B^A_\mu/\mathcal{J}^\lambda_{bd})=\mu$.  By Proposition \ref{Main tool} we conclude that $\mu\in Sp_T(\mathbb{P})\cap (\lambda,2^\lambda]$.
\end{proof}
The second corollary provides more restrictions on the $\sup(Sp_T(\mathbb{P}))$.
\begin{corollary}
    If $\sup(Sp_T(\mathbb{P}))$ is a singular cardinal, then $cf(\sup(Sp_T(\mathbb{P})))\in Sp_T(\mathbb{P})$.
\end{corollary}
Using Shelah's representation theorem we can prove something slightly more accurate:
\begin{corollary}
    Suppose that $\lambda$ is a singular limit cardinal of uncountable cofinality such that $\cf(\lambda)\notin Sp_T(\mathbb{P})$ and for some club $C\subseteq \lambda$, $C^{(+)}:=\{c^+\mid c\in C\}\subseteq Sp_T(\mathbb{P})$. Then $\lambda^+\in Sp_T(\mathbb{P})$. 
\end{corollary}
\begin{proof}
    Let  $cf(\lambda)=\theta<\lambda$. Fix any  $C\subseteq Sp_T(\mathbb{P})$ and suppose without loss of generality that $C$ has order type $\theta$. By Shelah's representation theorem (see  \cite[Thm. 2.23]{AbrahamMagidor2010}), there is a club $C_0\subseteq C$ such that $tcf(\prod C_0^{(+)}/\mathcal{J}^\lambda_{bd})=\lambda^+$. Since $C_0^{(+)}\subseteq Sp_T(\mathbb{P})$, we may apply Proposition \ref{Main tool} to conclude that $\lambda^+\in Sp_T(\mathbb{P})$.
\end{proof}

Moving to the case of regular cardinals, Proposition \ref{Main tool} yield the following:
\begin{corollary}
    If $\sup(Sp_T(\mathbb{P}))$ an is (weakly) inaccessible then for any unbounded set $B\subseteq Sp_T(\mathbb{P})$, $\mathfrak{d}(\prod B/\mathcal{J}_{bd})>\mathfrak{b}(\prod B/\mathcal{J}_{bd})$. In particular $2^{\sup(Sp_T(\mathbb{P}))}>(\sup(Sp_T(\mathbb{P})))^+$. 
\end{corollary}
\begin{remark}
    {}\ Suppose that $\lambda$ is strongly inaccessible and $B\subseteq Sp_T(\mathbb{P})$. Then: \begin{enumerate}
        \item It is possible that $\mathfrak{b}(\prod B/J)<\mathfrak{d}(\prod B/J)$, for example upon adding to a model of $\mathsf{GCH}$ $\mu$-many Cohen functions to $\lambda$, $\{f_i:B\to\lambda\mid i<\mu\}$ for some $\mu\geq cf(\mu)>\lambda$, define $f^*_i(b)=\begin{cases}
            0 & f_i(b)\geq b\\
            f_i(b) & \text{otherwise}
        \end{cases}$.  Then a standard chain condition argument shows that the set $\{f^*_i\mid i<\mu\}$ cannot be dominated by fewer than $cf(\mu)$-many functions. Hence in the extension $\mathfrak{d}(\prod B/\mathcal{J}_{bd})=cf(\mu)$. On the other hand, the usual argument showing that $\mathfrak{b}_\lambda=\lambda^+$ after Cohen forcing can be slightly modified to show $\mathfrak{b}(\prod B/\mathcal{J}^\lambda_{bd})=\lambda^+$.
        \item It is also possible that $\mathfrak{b}(\prod B/\mathcal{J}^\lambda_{bd})=\mathfrak{d}(\prod B/\mathcal{J}^\lambda_{bd})<2^\lambda$, for that we start with a model of GCH, and iterate with $<\lambda$-support, $\aleph_{\lambda^+}$-many times, the suitable Hechler forcing: conditions are of the form $(\varphi,f)$, where $\varphi\in \prod B\cap\mu$ for some $\mu<\lambda$ and $f\in\prod B\setminus \mu$. $\varphi$ is an initial segment of a generic function $\varphi_G\in\prod B$ which $\leq_{\mathcal{J}^\lambda_{bd}}$-dominates every ground model function $g\in \prod B$. The order is defined by $(\varphi,f)\leq (\varphi',f')$ if $\varphi'\subseteq \varphi$ (i.e. end-extension) and for every $b\in B\setminus \mu$, $f'(b)\leq f(b)$. See \cite[Definition 2.3]{GartiShelah} 
    \end{enumerate}
        
\end{remark} 

\section{Two generalizations}
\begin{definition}
    Let $\mathbb{P}$ be a directed set and $A\subseteq \mathbb{P}$. Define $$J_{bd}(A,\mathbb{P})=\{X\subseteq A\mid X\text{ is bounded in }\mathbb{P}\}.$$
    Let $J_{bd}(\mathbb{P})=J_{bd}(\mathbb{P},\mathbb{P})$.
\end{definition}
\begin{proposition}
    {}\ 
    \begin{enumerate}
        \item $J_{bd}(A,\mathbb{P})$ is an ideal on $A$.
        \item $J_{bd}(A,\mathbb{P})$ is proper if and only if $A$ is unbounded in $\mathbb{P}$.
        \item $(J_{bd}(A,\mathbb{P}),\subseteq)\leq_T \mathbb{P}$.
        \item $(J_{bd}(\mathbb{P}),\subseteq)\equiv_T\mathbb{P}$.
    \end{enumerate}
\end{proposition}
\begin{proof}
    $(1)$ easily follows from the directedness of $\mathbb{P}$. Also $(2)$ follows directly from the definition.  For $(3)$, this is Isbell's argument, let $F:J_{bd}(A,\mathbb{P})\to \mathbb{P}$ be defined by $F(X)=p$ where $p$ is some bound for $X$. We claim that $F$ is unbounded. If $\mathcal{A}\subseteq J_{bd}(A,\mathbb{P})$ is unbounded, then $\bigcup \mathcal{A}\notin J_{bd}(A,\mathbb{P})$. If $F[\mathcal{A}]$ was bounded by some $p$, then $p$ would also be a bound for $\bigcup\mathcal{A}$, contradiction.  
    To see $(4)$ we will prove that for $A=\mathbb{P}$, $F$ is also cofinal. Let $\mathcal{B}\subseteq J_{bd}(\mathbb{P})$ be cofinal and let $p\in\mathbb{P}$. Consider $E=\{q\in\mathbb{P}\mid q\leq p\}\in J_{bd}(\mathbb{P})$. Since $\mathcal{B}$ is cofinal, there is $B\in \mathcal{B}$ such that $E\subseteq B$. Hence $p\in B$. Since $F(B)$ is above all $b\in B$, $p\leq F(B)$, as wanted.
\end{proof}
From $(3)$ we get:
\begin{corollary}
    $Sp_T(J_{bd}(A,\mathbb{P}))\subseteq Sp_T(\mathbb{P})$
\end{corollary}
We can also say something about the uniformity of these ideals. Given a collection of sets $\mathcal{M}$ we let $\delta_{\mathcal{M}}$ be the minimal cardinality of a set $X\in \mathcal{M}$. In particular $[X]^{<\delta_\mathcal{M}}\cap \mathcal{M}=\emptyset$ and and $[X]^{\leq \delta_\mathcal{M}}\cap \mathcal{M}\neq\emptyset$. In particular, for an ideal $I$ over $Y$ we obtain the cardinals  $\delta_{I^+}$ and $\delta_{I^*}$, which satisfy $\delta_{I^+}\leq \delta_{I^*}\leq |Y|$. 
Note that if $A\in I^+$ is such that $|A|=\delta_{I^+}$, then $\delta_{I[A]^+}=\delta_{I[A]^*}$. It is not hard to see that $\delta_{J_{bd}(\mathbb{P})^+}$ is exactly the directness of $\mathbb{P}$.

\begin{proposition}
    Let $\theta<\chi(\mathbb{P})$ be any cardinal. Suppose that $\l A_\alpha\mid \alpha<\theta\r$ is a $\theta$-sequence of subsets of $\mathbb{P}$ of size $<\chi(\mathbb{P})$ such that $\bigcup_{\alpha<\theta}A_\alpha$ is cofinal in $\mathbb{P}$. Then $\sup\{\delta_{(J_{bd}(A_\alpha,\mathbb{P}))^*}\mid \alpha<\theta\}=\chi(\mathbb{P})$. 
\end{proposition}
\begin{proof}
    By the assumption of the Proposition, $\delta_{(J_{bd}(A_\alpha,\mathbb{P}))^*}\leq |A_\alpha|< \chi(\mathbb{P})$ and therefore $\sup\{\delta_{J_{bd}(A_\alpha,\mathbb{P})}\mid \alpha<\theta\}\leq\chi(\mathbb{P})$. Suppose towards a contradiction that $\sup\{\delta_{(J_{bd}(A_\alpha,P))^*}\mid \alpha<\theta\}\leq \mu^*<\chi(\mathbb{P})$. For each $\alpha<\theta$ let $X_\alpha\subseteq A_\alpha$ such that $|X_\alpha|\leq \mu^*$ and $A_\alpha\setminus X_\alpha$ is bounded, say, by $p_\alpha\in \mathbb{P}$. Consider $\mathcal{C}=\bigcup_{\alpha<\theta}(A_\alpha\setminus X_\alpha)\cup\{p_\alpha\}$. Then Clearly, $|\mathcal{C}|\leq\theta\cdot\mu<\chi(\mathbb{P})$, by every element of $p\in \bigcup_{\alpha<\theta}A_\alpha$ is dominated by some member of $\mathcal{C}$. This is a contradiction to the minimality of $\chi(\mathbb{P})$.
\end{proof}
    The previous proposition is useful in the case where $\chi(\mathbb{P})$ is a singular cardinal of uncountable cofinality (e.g. if $\mathbb{P}=(\mathcal{U},\supseteq)$). Take any cofinal sequence $\l p_\alpha\mid \alpha<\chi(\mathbb{P})\r$ witnessing  $\chi(\mathbb{P})$, and suppose that $(\theta_i)_{i<cf(\chi(\mathbb{P}))}$ is increasing continuous and cofinal in $\chi(\mathbb{P})$. Then for every $i<cf(\chi(\mathbb{P}))$, we let $A_i=\{p_\alpha\mid \alpha<\theta_i\}$. By F\"{o}dor's theorem and the previous proposition, we must have that for club many $i$'s, $\delta_{(J_{bd}(A_i,\mathbb{P}))^*}=|A_i|$.

Next, let us isolate some generalizations of the results from the previous section in more general settings. Let $\{\mathbb{P}^p\mid p\in\mathbb{P}\}$ be directed sets and $\mathbb{Q}$ be a directed set such that for each $p\in\mathbb{P}$, $\mathbb{P}^p\leq_T \mathbb{Q}$ as witnessed by $\pi^p:\mathbb{P}^p\to \mathbb{Q}$. Also
Suppose also that $\mathbb{P}\not\leq_T \mathbb{Q}$ and consider $$(\prod_{p\in\mathbb{P}}\mathbb{P}^p/J_{bd}(\mathbb{P}),\leq_{J_{bd}(\mathbb{P})})$$
Define for each $\vec{a}\in \prod_{p\in\mathbb{P}}\mathbb{P}^p$ an element $F(\vec{a})\in \mathbb{Q}$ as follows:
Consider $f^{\vec{a}}:\mathbb{P}\to\mathbb{Q}$ defined by $f^{\vec{a}}(p)=\pi^p(\vec{a}(p))$. Since $\mathbb{P}\not\leq_T\mathbb{Q}$, $f^{\vec{a}}$ is not unbounded which means that there is an unbounded set $I_{\vec{a}}\in J_{bd}(\mathbb{P})^+$ such that $f^{\vec{a}}[I_{\vec{a}}]$ is bounded by some $F(\vec{a})\in\mathbb{Q}$. 
\begin{theorem}
    $F$ maps a cofinal subset of $(\prod_{p\in\mathbb{P}}\mathbb{P}^p/J_{bd}(\mathbb{P}),\leq_{J_{bd}(\mathbb{P})})$ to an unbounded subset of $\mathbb{Q}$.
\end{theorem}
\begin{remark}
    Proposition \ref{Main tool} is a special case of the above theorem when $\mathbb{P}=cf(\lambda)$ and $\mathbb{P}_i=\lambda_i$ for some unbounded sequence $\lambda_i$ in $\lambda$. Indeed, in this case $J_{bd}(\mathbb{P})=\mathcal{J}^\lambda_{bd}$, and the requirament that the true cofinality exists simply means that $(\prod_{p\in\mathbb{P}}\mathbb{P}^p/J_{bd}(\mathbb{P}),\leq_{J_{bd}(\mathbb{P})})$ is a linear order which makes unbounded and cofinal sequence to be the same. Hence $F$ would be an unbounded map in this case. 
\end{remark}
\begin{proof}
    Suppose otherwise that $X\subseteq \prod_{p\in\mathbb{P}}\mathbb{P}^p$ is $\leq_{J_{bd}(\mathbb{P})}$-cofinal but $F[X]$ was bounded by $q^*$.
    We claim that there is $p\in\mathbb{P}$ such that $E_p:=\{\vec{a}(p)\mid p\in I_{\vec{a}}, \vec{a}\in X\}$ is unbounded in $\mathbb{P}^p$. This will yield a contradiction as for any  $\vec{a}(p)\in E_p$, the fact that $p\in I_{\vec{a}}$,  implies that $\pi^p(\vec{a}(p))\leq F(\vec{a})$. Also since $\vec{a}\in X$, $F(\vec{a})\leq q^*$, contradicting the fact that $\pi^p$ in an unbounded map. 

    So assume otherwise that for every $p\in\mathbb{P}$, $\{\vec{a}(p)\mid p\in I_{\vec{a}}, \vec{a}\in X\}$ is bounded in $\mathbb{P}^p$ by $b_p$, and we may assume that $b_p\notin \{\vec{a}(p)\mid p\in I_{\vec{a}}, \vec{a}\in X\}$ (for example if $\mathbb{P}^p$ does not have a maximal element) let $\vec{b}=\langle b_p\mid p\in\mathbb{P}\rangle$. Since $X$ is cofinal, there is $\vec{b}\leq_{J_{bd}}\vec{b}'\in X$. By the definition of $J_{bd}$, this means that there is $p^*\in\mathbb{P}$ such that whenever $p\not\leq p^*$, $b_p\leq b'_p$. Since $I_{\vec{b}'}$ is unbounded, there is $p_0\in I_{\vec{b}'}$ such that $p_0\not\leq p^*$ which means that $b_{p_0}\leq b'_{p_0}$. On the other hand, since $p_0\in I_{\vec{b}'}$ and $\vec{b}'\in X$, $b'_{p_0}\leq b_{p_0}$. Hence $b'_{p_0}=b_{p_0}$ contradicting the assumption that $b_{p_0}\notin \{\vec{a}(p_0)\mid p_0\in I_{\vec{a}}, \vec{a}\in X\}$.
\end{proof}
The following definition generalizes the notion of $(\lambda,\lambda)$-cohesive from \cite{Kanamori1978} and more generally, the equivalent condition for $\lambda\in Sp_T(\mathbb{P})$ (i.e. $\lambda\leq_T\mathbb{P}$).  
\begin{definition}
Let $\mathcal{I}$ be an ideal and let $\text{Space}(\mathcal{I})=\bigcup\mathcal{I}$. We say that $\mathbb{P}$ is $\mathcal{I}$-cohesive if there is a sequence $\l p_i\mid i\in \text{Space}(\mathcal{I})\r$ such that for every $S\in \mathcal{I}^+$ (i.e. $S\in P(\text{Space}(\mathcal{I}))\setminus \mathcal{I}$), $\{p_i\mid i\in S\}$ is unbounded in $\mathbb{P}$.
\end{definition}
Clearly, $\lambda\in Sp_T(\mathbb{P})$ if and only if $\mathbb{P}$ is $\mathcal{J}^\lambda_{bd}$-cohesive and more generally:
\begin{fact}
    Suppose that $cf(\mu)=\mu\leq\lambda$ and $\chi([\lambda]^{<\mu},\subseteq)=\lambda$. Then $[\lambda]^{<\mu}\leq_T \mathbb{P}$ if and only if $\mathbb{P}$ is $[\lambda]^{<\mu}$-cohesive.
\end{fact}
\begin{proposition}
    Let $\mathcal{I}$ be an ideal such that $\mathcal{I}\leq_T\mathbb{P}$. Then $\mathbb{P}$ is $\mathcal{I}$-cohesive. 
\end{proposition}
\begin{proof}
    Let $f:\mathcal{I}\to \mathbb{P}$ be unbounded. For each $\alpha\in \text{Space}(\mathcal{I})$ let $f(\{\alpha\})=p_\alpha$. If $S\in\mathcal{I}^+$, then $\{\{\alpha\}\mid \alpha\in S\}\subseteq\mathcal{I}$ is unbounded, and therefore $f[\{\{\alpha\}\mid\alpha\in S\}]=\{p_\alpha\mid \alpha\in S\}$ is unbounded. Hence $\mathbb{P}$ is $\mathcal{I}$-cohesive.
\end{proof}
The following monotonicity facts are easily verifiable:
\begin{fact}
    Suppose that $\mathcal{I}\supseteq \mathcal{J}$. If $\mathbb{P}$ is $\mathcal{J}$-cohesive then $\mathbb{P}$ is $\mathcal{I}$-cohesive.
\end{fact}
\begin{fact}
    If $\mathbb{P}\leq_T\mathbb{Q}$, then if $\mathbb{P}$ is $\mathcal{I}$-cohesive then also $\mathbb{Q}$ is $\mathcal{I}$-cohesive.
\end{fact}
Let us show that the character is also an upper bound for $\lambda$'s for which $\mathbb{P}$ is $\mathcal{I}$-cohesive:
\begin{proposition}
    Suppose that $\mathcal{I}$ is a $\chi(\mathbb{P})^+$-complete. Then $\mathbb{P}$ is not $\mathcal{I}$-cohesive. In particular if $\mathbb{P}$ is $\mathcal{I}$-cohesive for some $\lambda$-complete ideal over $\lambda$, then $\lambda\leq\chi(\mathbb{P})$.
\end{proposition}
\begin{proof}
    Let $\l b_\alpha\mid \alpha<\chi(\mathbb{P})\r$ be a cofinal subset of $\mathbb{P}$ and let $\l p_\alpha\mid \alpha<\lambda\r$ be any sequence of conditions in $\mathbb{P}$. Then for each $\alpha<\lambda$ there is $\beta<\chi(\mathbb{P})$ such that $p_\alpha\leq b_\beta$. Since $\mathcal{I}$ is $\chi(\mathbb{P})^+$-complete, there is $\beta<\chi(\mathbb{P})$ and a positive set $S\in \mathcal{I}^+$ such that for every $\alpha\in S$, $p_\alpha\leq b_\beta$. Hence $\mathbb{P}$ is not $\mathcal{I}$-cohesive.
\end{proof}
Let us denote by $Sp^*_T(\mathbb{P})$ the set which consist of all regular cardinals $\lambda$ such that there is a $\lambda$-complete ideal $\mathcal{I}$ on $\lambda$ for which $\mathbb{P}$ is $\mathcal{I}$-cohesive. 
\begin{corollary}
    $Sp_T(\mathbb{P})\subseteq Sp^*_T(\mathbb{P})\subseteq \chi(\mathbb{P})$.
\end{corollary}
The next theorem shows that a reasonable strategy to address Question \ref{Queation1} would be to show that $\sup(Sp^*_T(\mathbb{P}))=\chi(\mathcal{U})$.

\begin{theorem}
    Let $\lambda$ be a regular cardinal, $\mathcal{U}$ be an ideal on $\lambda$ and $\mathcal{U}$ be an ultrafilter. Then the following are equivlanet.
    \begin{enumerate}
        \item $\mathcal{U}$ is $\mathcal{I}$-cohesive.
        \item there is a cover $X\in M_{\mathcal{U}}$ for $j_U''\lambda$ such that for any $S\in I_\lambda^+$, $j_U(S)\not\subseteq X$.
    \end{enumerate}
\end{theorem}
\begin{proof}
    Suppose that $\mathcal{U}$ is $\mathcal{I}$-cohesive and let $\l A_i\mid i<\lambda\r\subseteq\mathcal{U}$ witness this. For each $\alpha<\kappa$, let $f(\alpha)=\{i<\lambda\mid \alpha\in A_i\}$ and let $X=[f]_{\mathcal{U}}$. We claim that $X$ is as wanted. Indeed, for every $i<\lambda$, $\{\alpha<\lambda\mid i\in f(\alpha)\}=A_i\in\mathcal{U}$. By \L o\'{s} Theorem $M_{\mathcal{U}}\models j_{\mathcal{U}}(i)\in [f]_{\mathcal{U}}$. Suppose that $S\in \mathcal{I}^+$ and note that
    \begin{equation}\label{equaiton2}
        \bigcap_{i\in S}A_i=\{\alpha<\kappa\mid S\subseteq f(\alpha)\}
    \end{equation}
    Since $\{A_i\mid i\in S\}$ is unbounded, $\bigcap_{i\in S}A_i\notin \mathcal{U}$, hence by \L o\'{s} Theorem $M_{\mathcal{U}}\models j_{\mathcal{U}}(S)\not\subseteq X$. 
    
    In the other direction, suppose that $X\in M_{\mathcal{U}}$ is a a cover as in $(2)$. Let $f$ be a representative function $[f]_{\mathcal{U}}=X$. Set $A_i=\{\alpha\mid i\in f(\alpha)\}$. Since $M_{\mathcal{U}}\models i\in X$, $A_i\in\mathcal{U}$. Given any $S\in \mathcal{I}^+$, Equation \ref{equaiton2} still holds true, and again use $M_{\mathcal{U}}\models j_{\mathcal{U}}(X)\not\subseteq S$ to see that$\bigcap_{i\in S}A_i\notin \mathcal{U}$. Hence $\l A_i\mid i<\lambda\r$ witnesses that $\mathcal{U}$ is $\mathcal{I}$-cohesive.
\end{proof}
\begin{theorem}
    Suppose that $B\subseteq Sp_T(\mathbb{P})$ and:
    \begin{enumerate}
        \item $cf(|B|)=|B|$ and let $\mathcal{I}$ be an ideal on $|B|$ such that $\mathbb{P}$ is not $\mathcal{I}$-cohesive.
        \item $tcf(\prod B /\mathcal{I})=\mu$.
    \end{enumerate}
    Then $\mu\in Sp_T(\mathbb{P})$.
\end{theorem}
\begin{proof}
The proof is a generalization of \ref{Main tool}.
    Let $\l f_\beta\mid \beta<\mu\r$ be a scale in $\prod B/\mathcal{I}$. and $\l p^b_\alpha\mid \alpha<b\r$ witness $\mathbb{P}$ is $b\in Sp_T(\mathbb{P})$. For every $\beta<\mu$, consider $\{p^b_{f_\beta(b)}\mid b\in B\}$. Since $\mathbb{P}$ is note $\mathcal{I}$-cohesive There is $I_\beta\in \mathcal{I}^+$ and $p^*_\beta$ such that for every $b\in I_\beta$, $p^b_{f_\beta(b)}\leq p^*_\beta$ bounding things. We claim that $\l p^*_\beta\mid \beta<\mu\r$ witnesses that $\mu\in Sp_T(\mathbb{P})$. Otherwise, let $J\in [\mu]^\mu$ and $p^*$ be such that for every $\beta\in J$, $p^*_\beta\leq p^*$. 
    
    As before, we claim that there is $b\in B$ such that $E_b=\{f_\beta(b)\mid \beta\in J, \ b\in I_\beta\}$ is unbounded. Otherwise, suppose that $\sup E_b\leq g(b)<b$. Since $\langle f_\beta\mid \beta\in J\r$ is dominating, there is $f_\beta$ for some $\beta\in J$ such that $g\leq_{\mathcal{I}} f_\beta$. Since $I_\beta\in \mathcal{I}^+$, there is $b\in I_\beta$ such that $g(\beta)\leq f_\beta(b)$. We conclude that $f_\beta(b)<f_\beta(b)$, contradiction. 
    \subsection{Further Directions and Open Problems}
    The following question seek for $\mathsf{ZFC}$ theorem:
    \begin{question}
        Suppose that $\lambda$ is a singular cardinal such that $cf(\lambda)\notin Sp_T(\mathbb{P})$ and $Sp_T(\mathbb{P})$ is unbounded in $\lambda$. Must $\lambda^+$ be a member of $Sp_T(\mathbb{P})$?
    \end{question}
    \begin{question}
        Given an ideal $\mathcal{I}$. Is there a poset $\mathbb{P}_{\mathcal{I}}$ such that for every directed set $\mathbb{Q}$, $\mathbb{P}_{\mathcal{I}}\leq \mathbb{Q}$ if and only $\mathbb{Q}$ is $\mathcal{I}$-cohesive. 
    \end{question}
    \begin{question}
        Is $Sp^*_T(\mathbb{P})=Sp_T(\mathbb{P})$?
    \end{question}
    \begin{question}
        Is $\sup(Sp^*_T(\mathcal{U}))=\chi(\mathcal{U})$?
    \end{question}
    \subsection*{Aknowledgemnts}
    The author would like to thank James Cummings for kindly explaining to him some relevant concepts in PCF Theory. Also to the members of the Univeristy of Toronto 
\end{proof}
\bibliographystyle{amsplain}
\bibliography{ref}

\end{document}